\documentclass[11pt]{article}
\usepackage{amssymb,amsmath}
\usepackage{float}
\usepackage{color,fullpage}
\usepackage{amsthm}
\newcommand{\RR}{\mathbb{R}}

\DeclareMathOperator*{\Min}{minimize}

\newcommand{\st}{{\quad\text{subject~to}~}}

\newtheorem{theorem}{Theorem}

\newtheorem{lemma}{Lemma}
\newtheorem{remark}{Remark}
\newtheorem{definition}{Definition}

\begin{document}

\title{The restricted strong convexity revisited:\\Analysis of equivalence to error bound and quadratic growth}

\author{Hui Zhang\thanks{
College of Science, National University of Defense Technology,
Changsha, Hunan, 410073, P.R.China. Corresponding author. Email: \texttt{h.zhang1984@163.com}}
}



\date{}

\maketitle

\begin{abstract}
The restricted strong convexity is an effective tool for deriving globally linear convergence rates of descent methods in convex minimization. Recently, the global error bound and  quadratic growth properties appeared as new competitors. In this paper, with the help of Ekeland's variational principle, we show the equivalence between these three notions. To deal with convex minimization over a closed convex set and structured convex optimization, we propose a group of modified versions and a group of extended versions of these three notions by using gradient mapping and proximal gradient mapping separately, and prove that the equivalence for the modified and extended versions still holds. Based on these equivalence notions, we establish new asymptotically linear convergence results for the proximal gradient method.  Finally, we revisit the problem of minimizing the composition of an affine mapping with a strongly convex differentiable function over a polyhedral set, and obtain a strengthened property of the restricted strong convex type under mild assumptions.
\end{abstract}

\textbf{Keywords:} restricted strong convexity, global error bound, quadratic growth property, gradient mapping, linear convergence


\section{Introduction}
To obtain globally linear convergence rates of gradient-type methods for minimizing convex (not necessarily strongly convex) differentiable functions, we recently proposed the restricted strong convexity (RSC) \cite{zhang2013gradient,zhang2015restricted}, which is a strictly weaker concept than the strong convexity.  Up to now, it has been proved that the RSC property is a very powerful tool for analyzing descent methods including (in)exact gradient method, restarted nonlinear CG, BFGS and its damped limited memory variants L-D-BFGS \cite{Frank2015linear,zhang2015restricted}.  Almost parallel to work \cite{zhang2013gradient}, the authors of \cite{wang2013iteration,man2013non} defined the global error bound (GEB) property in the spirit of Hoffman's celebrated result on error bounds for systems of linear inequalities \cite{Hoffman1952approximate,luo1992linear}. They showed that the GEB property also guarantees globally linear convergence results for descent methods. Moreover, they figured out a class of convex programs that frequently appear in machine learning obeying the GEB property. Very recently, the authors of \cite{necoara2015linear,liu2015asyn, gong2015linear} proposed the quadratic growth (QG) property with different names; it was called second order growth property in \cite{necoara2015linear}, optimal strong convexity in  \cite{liu2015asyn}, and semi-strongly convex property in \cite{gong2015linear}. They showed that the QG property can guarantee globally linear convergence results for descent methods as well. Since each of these three notions contributes as a linear convergence guarantee, it should be interesting to see what is the relationship between them. With the help of Ekeland's variational principle, we show that they are actually equivalent.

To deal with convex minimization over a closed convex and structured convex optimization, we propose a group of modified versions and a group of extended versions of these three notions by using gradient mapping and proximal gradient mapping separately \cite{nesterov2004introductory} to replace the gradient notion. Similarly, the modified and extended versions can be used to derive globally linear convergence results for a large class of descent methods in convex minimization \cite{necoara2015linear, Drusvyatskiy2016error,gong2015linear,beck2015linearly}. If the objective function in convex program involves the gradient-Lipschitz-continuous property, we prove that the equivalence for the modified and extended versions still holds. Based on these equivalence notions, we establish new asymptotically linear convergence results for the proximal gradient method, that are complementary to recently appearing theory.

The equivalence results in this paper provide us with alternative ways to check whether a given convex minimization problem satisfies the RSC property; in some cases, to check one of the equivalence properties might be much easier than to check the others. As a case study, we investigate the problem of minimizing the composition of an affine mapping with a strongly convex differentiable function over a polyhedral set, which is very popular in machine learning. We prove that this problem enjoys a strengthened property of the RSC type and hence the modified GEB property without any compactness assumption of polyhedral sets.

At the time of writing this paper, the authors of \cite{bolte2015error} showed the equivalence between the error bound (corresponding to our growth property) and the Kurdyka-${\L}$ojasiewicz inequality for convex functions having a moderately flat profile near the set of minimizers. As the convex functions are differentiable, this paper might provide complementary results to that in \cite{bolte2015error}. When our paper was under review, the authors of \cite{Drusvyatskiy2016error} posted their paper concerning the equivalence between the error bound and quadratic growth properties on arXiv. They defined the error bound condition by using the proximal gradient mapping, that generalizes our modified error bound property and the global error bound from the beginning proposed in  \cite{wang2013iteration}. However, they did not discuss the equivalence to the RSC property. Besides, our convergence results for the proximal gradient method are new and might be of interest in themselves.

The rest of paper is organized as follows. In Section 2, we present the basic notations and concepts discussed in this paper. In Section 3, we analyze the relationship of different notions. In Section 4, we discuss the convergence of proximal gradient method implied by the equivalence notions. In Section 5 is devoted to a case study.

\section{Notation and definitions}
We denote by $d(x, \mathcal{Y})$ the distance from a point $x$ to a nonempty closed set $\mathcal{Y}$; that is, $d(x, \mathcal{Y})=\inf_{y\in \mathcal{Y}}\|x-y\|$. When $\mathcal{Y}$ is a single point set, i.e., $\mathcal{Y}=\{y\}$, we use $d(x,y)$ to replace $d(x, \mathcal{Y})$ for simplicity.  We will consider functions that take values in the extended real line $\overline{\RR}:=\RR \bigcup \{+\infty\}$. The projection of $x$ onto a nonempty closed convex set $\mathcal{Y}$ is denoted by $[x]_\mathcal{Y}^+$. The spectral norm of a matrix $X$ is given by $\|X\|$. The terminology below follows from \cite{nesterov2004introductory}. A convex differentiable function $g$ is gradient-Lipschitz-continuous if there exists a positive scalar $L$ such that
\begin{equation}
\|\nabla g(x)-\nabla g(y)\|\leq L\|x-y\|, \quad\forall x, y\in\RR^n,\label{Lip}
\end{equation}
or equivalently,
\begin{equation}
0\leq g(y)- g(x)-\langle \nabla g(x), y-x\rangle\leq \frac{L}{2}\|y-x\|^2, \quad\forall x, y\in\RR^n,
\end{equation}
and strongly convex if there exists a positive scalar $\mu$ such that
\begin{equation}
\langle \nabla g(x)-\nabla g(y), x-y\rangle \geq \mu \|x-y\|^2, \quad\forall x, y\in\RR^n.\label{SC}
\end{equation}

\subsection{Original versions for unconstrained convex program}
\begin{definition}\label{original}
Let $f:\RR^n\rightarrow \RR$ be convex differentiable. Denote $f^*=\min_{x\in \RR^n} f(x)$ and $\mathcal{X}=\arg\min_{x\in \RR^n} f(x)$ and assume that $\mathcal{X}$ is nonempty. Then the unconstrained convex program
 $$\Min_{x\in \RR^n} f(x)$$
 obeys
\begin{description}
  \item[(a)]  the restricted strongly convex property with constant $\nu>0$ if it satisfies the restricted secant inequality
\begin{equation}
\langle \nabla f(x), x-[x]_\mathcal{X}^+\rangle \geq \nu\cdot d^2(x,\mathcal{X}),\quad \forall x\in \RR^n.
\end{equation}
  \item[(b)]  the global error bound property with constant $\kappa>0$ if it satisfies the error upper bound inequality
  \begin{equation}
\|\nabla f(x)\| \geq \kappa\cdot d(x,\mathcal{X}),\quad  \forall x\in \RR^n.
\end{equation}
  \item[(c)] the quadratic growth property with constant $\tau>0$ if it satisfies the second order growth of the function value
\begin{equation}
 f(x)-f^*\geq \frac{\tau}{2}\cdot d^2(x,\mathcal{X}), \quad \forall x\in \RR^n.
\end{equation}
\end{description}
We use $RSC(\nu)$, $GEB(\kappa)$, and $QG(\tau)$ to stand for the defined properties above respectively.
\end{definition}
To ensure that $[x]_\mathcal{X}^+$ is well defined, we need $\mathcal{X}$ to be closed; this is implied by the differentiable of $f(x)$. 
\begin{remark}
The RSC property first appeared in \cite{lai2013augmented} as a restricted secant inequality. The authors of \cite{zhang2013gradient,zhang2015restricted} formally defined it and figured out a class of non-trivial RSC functions.

The authors of \cite{wang2013iteration,man2013non} proposed the GEB property in the spirit of Hoffman's error bounds \cite{Hoffman1952approximate}. The local version of GEB only implies asymptotic linear convergence rates \cite{luo1992linear}.

The QG property appeared in a couple of papers  \cite{necoara2015linear,liu2015asyn,gong2015linear} with different names. The equivalence between the RSC and the QG of convex differentiable functions was shown in  \cite{zhang2015restricted,Frank2015linear}.
\end{remark}

\subsection{Modified versions for constrained convex program}
To deal with constrained convex programs, we first introduce the gradient mapping \cite{nesterov2004introductory}.
\begin{definition}
Let $ \gamma>0$ be a fixed constant and $Q$ be a closed convex set, and let $\bar{x}\in \RR^n$. Denote
\begin{align}
\nonumber
&x_Q(\bar{x};\gamma)=\arg\min_{x\in Q}[f(\bar{x})+ \langle \nabla f(\bar{x}), x-\bar{x}\rangle +\frac{\gamma}{2}\|x-\bar{x}\|^2]\\
&G^f_Q(\bar{x};\gamma)=\gamma (\bar{x}-x_Q(\bar{x};\gamma)).\nonumber
\end{align}
We call $G^f_Q(\bar{x};\gamma)$ the gradient mapping of function $f$ on $Q$.
\end{definition}
When $Q=\RR^n$, we have that
$$x_Q(\bar{x};\gamma)=\bar{x}-\frac{1}{\gamma}\nabla f(\bar{x}) ~~\textrm{and}~~G^f_Q(\bar{x};\gamma)=\nabla f(\bar{x}).$$
The latter implies that the gradient mapping generalizes the gradient notion.

\begin{definition}\label{modified}
Let $f:\RR^n\rightarrow \RR$ be a convex differentiable function and let $Q$ be a nonempty closed convex set. Denote $\mathcal{X}=\arg\min_{x\in Q} f(x)$ and $f^*=\min_{x\in Q} f(x)$ and assume that $\mathcal{X}$ is  nonempty. Let $\gamma>0$ be a fixed constant. Then the constrained convex program
 $$\Min_{x\in Q} f(x)$$
 obeys
\begin{description}
  \item[(a)]  the modified restricted strongly convex property on $Q$ with constant $\nu>0$ if it satisfies the restricted secant inequality
\begin{equation}
\langle G^f_Q(y;\gamma), y-[y]_\mathcal{X}^+\rangle \geq \nu\cdot d^2(y,\mathcal{X}),\quad \forall y\in Q.
\end{equation}
  \item[(b)]  the modified global error bound property on $Q$ with constant $\kappa>0$ if it satisfies the error upper bound inequality
  \begin{equation}
\|G^f_Q(y;\gamma)\| \geq \kappa\cdot d(y,\mathcal{X}), \quad\forall y\in Q.
\end{equation}
  \item[(c)] the modified quadratic growth property on $Q$  with constant $\tau>0$ if it satisfies the second order growth of the function value
\begin{equation}
 f(y)-f^*\geq \frac{\tau}{2}\cdot d^2(y,\mathcal{X}), \quad\forall y\in Q.
\end{equation}
\end{description}
We use $mRSC(\nu)$, $mGEB(\kappa)$, and $mQG(\tau)$ to stand for the defined properties above respectively.
\end{definition}
Again, the closedness of $\mathcal{X}$ is guaranteed by the differentiable of $f(x)$ and hence $[y]_\mathcal{X}^+$ is well defined. 
\begin{remark}
The author of \cite{Frank2015linear} also proposed a modified RSC by considering a convex constraint set $Q$. But they required that both $X_f=\arg\min_{x\in \RR^n} f(x)$ and $Q\bigcap X_f$ are nonempty. Such assumptions are very strong conditions and many practical problems may fail to satisfy.
\end{remark}

\begin{remark}
When $Q=\RR^n$, the modified versions return to the corresponding original versions since $G^f_Q(\bar{x};\gamma)=\nabla f(\bar{x})$. Therefore, the modified Definition \ref{modified} can be viewed as a generalization of Definition \ref{original}.
\end{remark}

\subsection{Extended versions via proximal gradient mapping}
To introduce extended versions of the previous notions for structure convex optimization, we need the concept of proximal gradient mapping \cite{beck2009fast}.
\begin{definition}
Let $ \gamma>0$ be a fixed constant, $f: \RR^n\rightarrow \RR$ be a differentiable function, $g:\RR^n\rightarrow \overline{\RR}$ be a closed convex function, and $\bar{x}\in \RR^n$. Denote
\begin{align}
\nonumber
&p^f_g(\bar{x};\gamma)=\arg\min_{x}[f(\bar{x})+ \langle \nabla f(\bar{x}), x-\bar{x}\rangle +\frac{\gamma}{2}\|x-\bar{x}\|^2+g(x)]\\
&G^f_g(\bar{x};\gamma)=\gamma (\bar{x}-p^f_g(\bar{x};\gamma)).\nonumber
\end{align}
We call $G^f_g(\bar{x};\gamma)$ the proximal gradient mapping of functions $f$ and $g$.
\end{definition}
Let $Q$ be a closed nonempty convex set. When $g$ is the indicator function
 \begin{equation}
I_Q(x)=\left\{\begin{array}{rl} 0,&
x\in Q,\\+\infty,& x\notin Q,
\end{array} \right.  \nonumber
\end{equation}
we have that
$$p^f_g(\bar{x};\gamma)=x_Q(\bar{x};\gamma).$$
This implies that the proximal gradient mapping generalizes the gradient mapping.

\begin{definition}\label{extended}
Let $f:\RR^n\rightarrow \RR$ be a convex differentiable function and $g:\RR^n\rightarrow \overline{\RR}$ be a closed convex function. Denote $\mathcal{X}=\arg\min_{x} \varphi(x):= f(x)+g(x)$ and $\varphi^*=\min_{x} \varphi(x)$ and assume that $\mathcal{X}$ is nonempty. Let $\gamma>0$ be a fixed constant. Then the following convex program
 $$\Min_{x} \varphi(x)=f(x)+g(x)$$
 obeys
\begin{description}
  \item[(a)]  the extended restricted strongly convex property with parameter $\nu, \omega>0$ if it satisfies the restricted secant inequality
\begin{equation}
\langle G^f_g(y;\gamma), y-[y]_\mathcal{X}^+\rangle \geq \nu\cdot d^2(y,\mathcal{X}),\quad \forall y\in [\varphi\leq \varphi^*+\omega].
\end{equation}
  \item[(b)]  the extended global error bound property with  parameter $\kappa, \omega>0$ if it satisfies the error upper bound inequality
  \begin{equation}
\|G^f_g(y;\gamma)\| \geq \kappa\cdot d(y,\mathcal{X}), \quad\forall y\in [\varphi\leq \varphi^*+\omega].
\end{equation}
  \item[(c)] the extended quadratic growth property with constant  parameter $\tau, \omega>0$  if it satisfies the second order growth of the function value
\begin{equation}
 \varphi (y)-\varphi^*\geq \frac{\tau}{2}\cdot d^2(y,\mathcal{X}), \quad\forall y\in [\varphi\leq \varphi^*+\omega].
\end{equation}
\end{description}
We use $eRSC(\nu, \omega)$, $eGEB(\kappa, \omega)$, and $eQG(\tau, \omega)$ to stand for the defined properties above respectively.
\end{definition}
It is easy to see that $\varphi(x)$ is lower semicontinuous over $\RR^n$ and hence $\mathcal{X}$ is closed; see e.g. Lemma 2.6.3 in \cite{ruszczynski2006nonlinear}. This ensures that $[y]_\mathcal{X}^+$ is well defined.  
\begin{remark}
 The eQG appeared in \cite{gong2015linear} under the name of semi-strongly convex property. The authors of \cite{zhou2015unified} proposed an analog of the eGEB property and exploited it by borrowing tools from set-valued analysis.  The authors of \cite{Drusvyatskiy2016error} introduced the eGEB and proved that it is equivalent to the eQG. Our novelty here lies in the definition of eRSC.
\end{remark}

\section{Equivalence analysis}

\subsection{Equivalence among the original versions}

In what follows, we prove that the notions defined in Definition \ref{original} are actually equivalent. The idea of proof is mainly inspired by the seminal paper \cite{Artacho2008char} and heavily relies on the well-known Ekeland's variational principle in Lemma 1.

\begin{theorem}
Under the setting of Definition 1,  the restricted strongly convex property, the global error bound property, and the quadratic growth property are equivalent in the following sense:
$$QG(\nu)\Rightarrow RSC(\frac{\nu}{2})\Rightarrow GEB(\frac{\nu}{2}) \Rightarrow QG(\frac{\nu}{4}).$$
\end{theorem}
\begin{proof}
The implication of $QG(\nu)\Rightarrow RSC(\frac{\nu}{2})$ has been shown in  \cite{zhang2015restricted,Frank2015linear}. The implication $RSC(\frac{\nu}{2})\Rightarrow GEB(\frac{\nu}{2})$ is a direct consequence after applying the Cauchy-Schwartz inequality. It only needs to prove $GEB(\nu) \Rightarrow QG(\frac{\nu}{2})$. Now assume that $f$ has the GEB property  with constant $\nu>0$. It suffices to prove that for all $0\leq \alpha <\frac{1}{4}$, the following holds:
$$f(z)-f^*\geq \alpha \nu\cdot d^2(z,\mathcal{X}), \quad \forall z\in \RR^n.$$
If this is not true, then there must exist $z_0\in \RR^n$ such that
$$f(z_0)<f^*+\alpha \nu\cdot d^2(z_0,\mathcal{X}).$$
Clearly, $z_0\notin \mathcal{X}$ and hence $d(z_0, \mathcal{X})>0$ since $\mathcal{X}$ is a nonempty closed set. Let $\lambda =\frac{1}{2}d(z_0, \mathcal{X})$. By Ekeland's variational principle with $\epsilon=\alpha \nu\cdot d^2(z_0,\mathcal{X})=4\alpha \nu \lambda^2$, there exists $x_0\in \RR^n$ such that  $d(x_0,z_0)\leq \lambda$ and
$$f(x)\geq f(x_0)-\frac{\epsilon}{\lambda}d(x,x_0)=f(x_0)-4\alpha \nu \lambda \cdot d(x,x_0), \quad \forall x\in \RR^n.$$
Then, $x_0$ minimizes the convex function $f(x)+4\alpha \nu \lambda \cdot d(x,x_0)$. By the first-order optimality condition, we get
$$0\in  \nabla f(x_0)+4\alpha \nu \lambda \cdot \partial (\|\cdot-x_0\|)(x_0)=\nabla f(x_0)+4\alpha \nu \lambda \cdot \mathcal{Y},$$
where $\mathcal{Y}=\{y\in\RR^n: \|y\|\leq 1\}$ \cite{ruszczynski2006nonlinear}. Hence, we can find $y_0\in\mathcal{Y}$ such that $\nabla f(x_0)=-4\alpha \nu \lambda y_0$. Since
$$2\lambda =d(z_0, \mathcal{X}) \leq d(x_0, z_0)+d(x_0, \mathcal{X})\leq \lambda +d(x_0, \mathcal{X}),$$
we have $d(x_0, \mathcal{X})\geq \lambda$. Therefore,
$$\|\nabla f(x_0)\|=4\alpha \nu \lambda \|y_0\|\leq 4\alpha \nu \lambda\leq 4\alpha \nu \cdot d(x_0, \mathcal{X})< \nu\cdot d(x_0, \mathcal{X}),$$
which contradicts the GEB property. This completes the proof.
\end{proof}

\subsection{Equivalence among the modified and extended versions}
In this part, we deduce the equivalence among the modified and extended versions.

\begin{theorem}\label{main2}
 Let $f:\RR^n\rightarrow \RR$ be a convex differentiable function whose gradient is Lipschitz-continuous with positive scalar $L$ and $g:\RR^n\rightarrow \overline{\RR}$ be a closed convex function.  Let $\gamma \geq L$ be a fixed constant. Then the extended versions are equivalent in the following sense
$$eRSC(\nu_1, \omega)\Rightarrow eGEB(\kappa_1, \omega)\Rightarrow eQG(\tau_1,\omega)$$
and $$eQG(\tau_ 2,\omega)\Rightarrow eGEB(\kappa_2, \omega)\Rightarrow eRSC(\nu_2, \omega),$$
where $\omega\in (0, +\infty]$ is a fixed constant and the other parameters satisfy $\kappa_1=\tau_1=\nu_1$ and
$$\kappa_2=\frac{\tau_2 \gamma^2}{(2\gamma+\tau_2)(\gamma+L)}, ~~\nu_2=\frac{\kappa_2^2}{\gamma}.$$
In particular, the equivalence among the modified versions also holds in the same way by letting $\omega=+\infty$.
\end{theorem}
\begin{proof}
The equivalence of $eGEB(\kappa, \omega)$ and $eQG(\tau,\omega)$ has been shown in \cite{Drusvyatskiy2016error} recently; see Lemma \ref{impli}. Applying the Cauchy-Schwartz inequality to the left-hand side term of eRSC gives the eGEB property. It remains to prove that $eGEB$ implies $eRSC$.
 By using Lemma \ref{beck} with $x=[y]_\mathcal{X}^+$ and $\bar{x}=y$, we get
\begin{equation}
 \langle  G^f_g(y;\gamma),y-[y]_\mathcal{X}^+\rangle \geq \frac{1}{2\gamma}\|G^f_g(y;\gamma)\|^2 +\varphi(p^f_g(y;\gamma)) -\varphi([y]_\mathcal{X}^+).
\end{equation}
Noticing that $\varphi(p^f_g(y;\gamma)) - \varphi([y]_\mathcal{X}^+)=\varphi(p^f_g(y;\gamma)) - \varphi^*\geq 0$ and applying the eGEB property, we obtain
\begin{equation}
\langle  G^f_g(y;\gamma),y-[y]_\mathcal{X}^+\rangle\geq \frac{\kappa_2^2}{2\gamma}d^2(y,\mathcal{X}), \quad \forall y\in [\varphi\leq \varphi^*+\omega],
\end{equation}
which completes the proof.
\end{proof}

\section{Convergence analysis}
In this section, we discuss asymptotically linear convergence of the proximal gradient method, based on the equivalence notions defined before.

The proximal gradient method, also known as the forward-backward splitting method, is fundamental for the following structured optimization problem
$$\Min_{x} \varphi(x):= f(x)+g(x),$$
where $f$ is a convex function with the  gradient-Lipschitz-continuous property and $g$ is a closed convex function. It can be stated as
$$x_{k+1}=x_k-\frac{1}{\gamma} G^f_g(x_k;\gamma),$$
where the constant $\gamma>0$ is appropriately chosen. Based on the eGEB property, the authors of \cite{Drusvyatskiy2016error} observed that an asymptotically $Q$-linear convergence in function values can be assured for this method, and moreover, if the iterates $x_k$ have some limit point $x^*$, then $x_k$ asymptotically converge $R$-linearly. Motivated by the arguments in \cite{zhang2015restricted}, we have established below an asymptotically $Q$-linear convergence in the distance values  $d(x_{k},\mathcal{X})$, and proved that $x_k$ themselves asymptotically converge $R$-linearly to a limit point $x^*$ under a compactness assumption on the minimizer set $\mathcal{X}$. The following result is complimentary to that of \cite{Drusvyatskiy2016error} and might be of interest in itself.

\begin{theorem}\label{mainadd1}
 Let $f:\RR^n\rightarrow \RR$ be a convex differentiable function whose gradient is Lipschitz-continuous with positive scalar $L$ and $g:\RR^n\rightarrow \overline{\RR}$ be a closed convex function. Denote $\mathcal{X}=\arg\min_{x} \varphi(x):= f(x)+g(x)$ and $\varphi^*=\min_{x} \varphi(x)$ and assume that $\mathcal{X}$ is  nonempty. Let $\gamma\geq L$ be a fixed constant. Suppose that $\Min_{x} \varphi(x)$ satisfies the $eQG(\tau,\omega)$ property (equivalently, the $eGEB(\kappa, \omega)$ property). Then, the iterates $x_k$ generated by the proximal gradient method asymptotically converge $Q$-linearly, that is there exists an index $m$ such that the inequality
$$ d(x_{k+1},\mathcal{X})\leq \sqrt{\frac{\gamma}{\gamma + \tau}}d(x_k,\mathcal{X}) $$
holds for all $k\geq m$. Moreover, if $\mathcal{X}$ is compact, then the iterates $x_k$ asymptotically converge $R$-linearly to some limit point $x^*\in\mathcal{X}$ in the sense that
$$ d^2(x_{k+m},x^*)\leq C\cdot (1-\frac{\kappa}{2\gamma})^k$$
holds for all $k\geq 1$, where $m>\frac{L}{2\omega}d^2(x_0,\mathcal{X})$ and $C=\frac{2(\varphi(x_m)-\varphi^*)}{\gamma}(\sum_{i=0}^\infty (1-\frac{\kappa}{2\gamma})^{\frac{i}{2}})^2$.

\end{theorem}
The proof idea below is partially inspired by \cite{Drusvyatskiy2016error,zhang2015restricted}.
\begin{proof}
Let $m= \lceil \frac{L}{2\omega}d^2(x_0,\mathcal{X})\rceil$ where $\lceil x \rceil$ denotes the smallest integer larger than $x$. By using the standard sublinear estimate \cite{beck2009fast}
$$\varphi(x_k)-\varphi^*\leq \frac{L\cdot d^2(x_0,\mathcal{X})}{2k},$$
we deduce that $\varphi(x_k)-\varphi^*\leq \omega$ holds for all $k\geq m$. Denote the projection point of $x$  onto $\mathcal{X}$ by $x^{\prime}$ . By invoking Lemma \ref{beck} with $x=x_k^{\prime}$ and $\bar{x}=x_k$, we have that
\begin{equation}\label{bound1}
 \langle  G^f_g(x_k;\gamma),x_k-x_k^{\prime}\rangle \geq \frac{1}{2\gamma}\|G^f_g(x_k;\gamma)\|^2 +\varphi(x_{k+1}) -\varphi^* .
\end{equation}
Now, together with the $eQG(\tau,\omega)$ property, we derive that for all $k\geq m$
\begin{subequations}
\begin{align}
d^2(x_{k+1},\mathcal{X})= &\|x_{k+1}-x_{k+1}^{\prime}\|^2\leq \|x_{k+1}-x_k^{\prime}\|^2 \\
=& \|x_k-x_k^{\prime} -\frac{1}{\gamma}G^f_g(x_k;\gamma)\|^2 \\
=& \|x_k-x_k^{\prime}\|^2 -\frac{2}{\gamma}\langle  G^f_g(x_k;\gamma),x_k-x_k^{\prime}\rangle + \frac{1}{\gamma^2}\|G^f_g(x_k;\gamma)\|^2\\
\leq & \|x_k-x_k^{\prime}\|^2-\frac{2}{\gamma}(\varphi(x_{k+1}) -\varphi^* ) \\
\leq & d^2(x_k,\mathcal{X})-\frac{\tau}{\gamma} d^2(x_{k+1},\mathcal{X}),
\end{align}
\end{subequations}
which yields the asymptotically  $Q$-linear convergence of $x_k$.

To prove the convergence of $\{x_k\}$ themselves, we first consider the sequence $\{x^{\prime}_k\}\subseteq \mathcal{X}$, which must have a subsequence, denoted by $\{x^{\prime}_{k_i}\}$, converging to some point $x^*\in\mathcal{X}$ due to the compactness assumption on $\mathcal{X}$. We claim that $x^{\prime}_k\rightarrow x^*$ as $k\rightarrow \infty$. Otherwise, there must exist another subsequence of $\{x_k\}$, denoted by $\{x^{\prime}_{k_j}\}$, converging to a different point $\hat{x}\in\mathcal{X}$. Notice that
$$d(x^*,\hat{x})\leq d(x^*,x^{\prime}_{k_i})+d(x^{\prime}_{k_i}, x^{\prime}_{k_j})+d(x^{\prime}_{k_j},\hat{x})$$
and $d(x^{\prime}_{k_i}, x^{\prime}_{k_j})\leq d(x_{k_i}, x_{k_j}),$  where the latter follows from the nonexpansive property of projection operator. We get that
\begin{equation}\label{dist}
 d(x^*,\hat{x})\leq d(x^*,x^{\prime}_{k_i})+d(x^{\prime}_{k_j},\hat{x})+d(x_{k_i}, x_{k_j}).
\end{equation}
Denote $\ell_1=\min\{k_i, k_j\}$ and $\ell_2=\max\{k_i,k_j\}$; then
$$d(x_{k_i}, x_{k_j})\leq \sum_{i=\ell_1}^{\ell_2-1}d(x_i,x_{i+1})\leq \sum_{i=\ell_1}^{\infty}d(x_i,x_{i+1}).$$
By invoking Lemma \ref{beck} with $x=\bar{x}=x_i$, we have that
\begin{equation}\label{bound2}
\varphi(x_i)-\varphi(x_{i+1})\geq \frac{1}{2\gamma}\|G^f_g(x_i;\gamma)\|^2=\frac{\gamma}{2}\|x_i-x_{i+1}\|^2.
\end{equation}
On the other hand, inequality \eqref{bound1} implies
\begin{equation}\label{bound3}
\varphi(x_{i+1}) -\varphi^*\leq\| G^f_g(x_i;\gamma)\|^2\left( \frac{\|x_i-x^{\prime}_i\|}{\| G^f_g(x_i;\gamma)\|}-\frac{1}{2\gamma}\right).
\end{equation}
Applying the $eGEB(\kappa, \omega)$ property to \eqref{bound3} and combining with \eqref{bound2}, we get that
\begin{equation}\label{conv1}
\varphi(x_{i+1}) -\varphi^*\leq (1-\frac{\kappa}{2\gamma})(\varphi(x_i) -\varphi^*),~~i\geq m .
\end{equation}
Let $\ell_1\geq m$ and $i\geq m$; then
\begin{subequations}
\begin{align}
d^2(x_i,x_{i+1})= & \|x_i-x_{i+1}\|^2 \leq \frac{2}{\gamma}(\varphi(x_i)-\varphi(x_{i+1}))\leq  \frac{2}{\gamma}(\varphi(x_i)-\varphi^*)\\
\leq &  \frac{2}{\gamma}(1-\frac{\kappa}{2\gamma})^{i-m}(\varphi(x_m)-\varphi^*)
\end{align}
\end{subequations}
and hence
\begin{subequations}\label{bound4}
\begin{align}
d(x_{k_i},x_{k_j})\leq &\sum_{i=\ell_1}^{\infty}d(x_i,x_{i+1})\leq \sqrt{\frac{2(\varphi(x_m)-\varphi^*)}{\gamma}}\sum_{i=\ell_1}^\infty (1-\frac{\kappa}{2\gamma})^{\frac{i-m}{2}}\\
= &  D\cdot (1-\frac{\kappa}{2\gamma})^{\frac{\ell_1-m}{2}}\rightarrow 0, ~~as~~\ell_1\rightarrow \infty ,
\end{align}
\end{subequations}
where $D=\sqrt{\frac{2(\varphi(x_m)-\varphi^*)}{\gamma}}\sum_{i=0}^\infty (1-\frac{\kappa}{2\gamma})^{\frac{i}{2}}<\infty$. Together with the fact that $x^{\prime}_{k_i}\rightarrow x^*$ as $k_i\rightarrow \infty$ and $x^{\prime}_{k_j}\rightarrow \hat{x}$ as $k_j\rightarrow \infty$, we immediately get $d(x^*,\hat{x})=0$ by using \eqref{dist} with  $k_i\rightarrow \infty$ and $k_j\rightarrow \infty$, which contradicts $x^*\neq \hat{x}$. Therefore,  $x^{\prime}_k\rightarrow x^*$ as $k\rightarrow \infty$ indeed holds. Finally, in light of the asymptotically $Q$-linear convergence of $x_k$, we have that
$$d(x_k,x^*)\leq d(x_k,x^{\prime}_k)+ d(x^{\prime}_k,x^*)=d(x_k, \mathcal{X})+ d(x^{\prime}_k,x^*) \rightarrow 0, ~~as~~k\rightarrow\infty,$$
which implies that the iterates $x_k$ converge to $x^*\in \mathcal{X}$. The asymptotically $R$-linear convergence of $\{x_k\}$ follows by setting $k_i=\ell_1=k+m$ and letting $k_j\rightarrow +\infty$ in \eqref{bound4}. This completes the proof.
\end{proof}

\begin{remark}
Although \eqref{bound4} implies that $\{x_k\}$ is a Cauchy sequence and hence converges, it can not ensure its convergence to a point belonging to $\mathcal{X}$ without the compactness assumption on $\mathcal{X}$.
\end{remark}

\section{A case study: composition minimization with constraints}
The authors of \cite{zhou2015unified} presented a unified framework for establishing error bounds for a class of structured convex optimization problems. By the equivalence between the error bound condition and quadratic growth, the authors of \cite{Drusvyatskiy2016error} streamlined and illuminated the arguments in \cite{zhou2015unified} and also extended their results to a wider setting. Here, in a transparent way we derive a strengthened property of the RSC type for the following constrained convex program
\begin{equation}\label{cp}
\Min f(x),\quad\st~x\in Q,
\end{equation}
where $Q$ is a polyhedral set in the n-dimensional Euclidean space, and $f$ is the composition of an
affine mapping with a strongly convex differentiable function. We assume that $f$ and  $Q$ are of the special form:
 $$f(x)=g(Ex), ~~~Q=\{x\in \RR^n| Ax\leq b\},$$
where $E$ is some $m\times n$ matrix, $A$ is some $k\times n$ matrix, $b\in\RR^k$ is some vector, and $g$ is strongly convex and gradient-Lipschitz-continuous with positive scalars $\mu$ and $L$ respectively.

The convex program \eqref{cp} has been extensively studied for its wide applications in data networks and machine learning.

\begin{theorem}
Under the setting of Definition \ref{modified} and letting $\gamma\geq L\|EE^T\|$, we have that the constrained convex program \eqref{cp} obeys the strengthened  property of the RSC type:
\begin{equation}\label{smRSC}
 \langle G^f_Q(y;\gamma), y-[y]_\mathcal{X}^+\rangle\geq \frac{1}{2\gamma}\|G^f_Q(y;\gamma)\|^2+ C_1 \cdot d^2(y,\mathcal{X}), ~~\forall y\in Q,
\end{equation}
and the modified quadratic growth property:
\begin{equation}
 f(y)\geq f^*+ C_2\cdot d^2(y,\mathcal{X}), ~~\forall y\in Q,
\end{equation}
where $C_i, i=1,2$ are positive constant depending on matrices $E$ and $A$ and the positive scalar $\mu$.
\end{theorem}
\begin{proof}
By using \eqref{keyineq} in Lemma \ref{nest2} with $x=[y]_\mathcal{X}^+, \bar{x}=y$ and noticing that $f(x_Q(y;\gamma))\geq f([y]_\mathcal{X}^+)=f^*$, we get
\begin{equation}\label{Ersc}
 \langle G^f_Q(y;\gamma), y-[y]_\mathcal{X}^+\rangle\geq \frac{1}{2\gamma}\|G^f_Q(y;\gamma)\|^2+ \frac{\mu}{2} \|Ey-E[y]_\mathcal{X}^+\|^2.
\end{equation}
Since $g$ is strongly convex, there must exist a unique vector $t^*$ such that
$ Ex=t^*, \forall x\in \mathcal{X}$;  please refer to \cite{luo1992linear,wang2013iteration}.
Thus, $$\mathcal{X}=\{x| Ex=t^*\}\bigcap \{x|Ax\leq b\}.$$
Due to the result of Hoffman's error bound in Lemma \ref{hoffman}, there must exist a constant $\theta>0$ depending on matrices $E$ and $A$ such that for any $y\in Q=\{x|Ax\leq b\}$ it holds
$$\|Ey-E[y]_\mathcal{X}^+\|^2=\|Ey-t^*\|^2\geq \theta \| y-[y]_\mathcal{X}^+\|^2=\theta\cdot d^2(y,\mathcal{X}).$$
Thus, the strengthened  property of the RSC type follows from this and the inequality \eqref{Ersc}.

By Lemma \ref{opt}, for any $x^*\in\mathcal{X}$ we have $G^f_Q(x^*;\gamma)=0$ and hence $x^*=x_Q(x^*;\gamma)$. Therefore, using \eqref{keyineq} Lemma \ref{nest2} with $\bar{x}=[y]_\mathcal{X}^+$ and the fact of $\|Ey-E[y]_\mathcal{X}^+\|^2\geq \theta\cdot d^2(y,\mathcal{X})$, we get the modified quadratic growth property. This completes the proof.
\end{proof}

\begin{remark}
By Lemma A.8 in \cite{wang2013iteration}, we get
$$\|G^f_Q(y;\gamma)\|=\|\gamma (y-[y-\frac{1}{\gamma}\nabla f(y)]_\mathcal{X}^+)\|\leq \gamma \max(1,\gamma^{-1})\|y-[y-\nabla f(y)]_\mathcal{X}^+\|.$$
Thus, using this and applying the Cauchy-Schwartz inequality to \eqref{smRSC}, we get
$$\gamma \max(1,\gamma^{-1})\|y-[y-\nabla f(y)]_\mathcal{X}^+\|\geq C_1\cdot d(y,\mathcal{X}).$$
Or, equivalently there exists a constant $C_3>0$ such that
$$ C_3 \cdot d(y,\mathcal{X})\leq \|y-[y-\nabla f(y)]_\mathcal{X}^+\|.$$
On the contrast, the authors of \cite{wang2013iteration} derived the following error bound
$$  d(y,\mathcal{X})\leq \kappa_3\|y-[y-\nabla f(y)]_\mathcal{X}^+\|, ~~\forall y\in Q~~\textrm{and} ~~f(y)-f^*\leq M $$
for $f(x)=g(Ex)+q^Tx$. Although this bound enables the incorporation of the additional linear term, the constant $\kappa_3$ depends on the positive parameter $M$ and hence on the variable $y$. To avoid such dependence, the authors of \cite{beck2015linearly} derived the quadratic growth property
$$ f(y)\geq f^*+ \kappa_4\cdot d^2(y,\mathcal{X}), ~~\forall y\in Q$$
for $f(x)=g(Ex)+q^Tx$ by assuming that $Q$ is compact, where $\kappa_4>0$ is a constant. Here,
 we can drop the compactness assumption by neglecting the linear term $q^Tx$.
\end{remark}

\begin{remark}
There is an alternative way  in \cite{necoara2015linear} to prove the modified QG  property. Indeed,
By the strong convexity of $g$ and the fact of $\|Ey-E[y]_\mathcal{X}^+\|^2\geq \theta\cdot d^2(y,\mathcal{X})$, we derive for $\forall y\in Q$ that
$$f(y)-f^* = f(y)-f([y]_\mathcal{X}^+)=g(Ey) -g(E[y]_\mathcal{X}^+)  \geq \frac{\mu}{2}\|Ex-E[x]_\mathcal{X}^+\|^2\geq \frac{\mu\theta}{2}\cdot d^2(y,\mathcal{X}).$$
Thus, by Theorem 2 we immediately get the mRSC property
$$\langle G^f_Q(y;\gamma), y-[y]_\mathcal{X}^+\rangle\geq C_4\cdot d^2(y,\mathcal{X}),$$
where $C_4$ is some positive constant.
\end{remark}

\section*{Acknowledgements}
 We would like to thank anonymous reviewers for their valuable comments, with which great improvements have been made in this manuscript. The work is supported by the National Science Foundation of China (No.11501569 and No.61571008).

\section{Appendix}
\begin{lemma}[Ekeland's variational principle, \cite{Artacho2008char}]
Let $(X, d(\cdot,\cdot))$ be a complete metric space and let $f:\RR^n\rightarrow \RR\bigcup \{+\infty\}$ be a lower semicontinuous function bounded from below, where $d(x,y)$ stands for the Euclidean distance. Suppose that for some $\epsilon>0$ and $z\in X$, $f(z)<\inf f +\epsilon$. Then for any $\lambda>0$ there exist $y\in X$ such that $d(z,y)\leq \lambda$ and
$$f(x)+\frac{\epsilon}{\lambda}d(x,y)\geq f(y), ~\forall x\in X.$$
\end{lemma}

\begin{lemma}[Corollary 3.6, \cite{Drusvyatskiy2016error}]\label{impli}
Let $f$ be gradient-Lipschitz-continuous with positive scalar $L$ and $g:\RR^n\rightarrow \overline{\RR}$ be a closed convex function. Then the $eQG(\tau, \omega)$ implies $eGEB(\kappa, \omega)$ with $\kappa=\frac{\tau \gamma^2}{(2\gamma+\tau)(\gamma+L)}$. Conversely, the
$eGEB(\kappa, \omega)$ implies the $eQG(\tau, \omega)$ with any $\tau\in (0, \kappa)$.
\end{lemma}

The following result describes an important property of the proximal gradient mapping.
 \begin{lemma}[\cite{beck2009fast}]\label{beck}
Let $f(x)$ be gradient-Lipschitz-continuous with positive scalar L, $g:\RR^n\rightarrow \overline{\RR}$ be a closed convex function,  $\gamma\geq L$,  and $\bar{x}\in\RR^n$. Denote $\varphi(x)=f(x)+g(x)$. Then, for any $x\in Q$ we have
\begin{equation}\label{keyineq}
\varphi(x)-\varphi(p^f_g(\bar{x};\gamma))\geq  \langle G^f_g(\bar{x};\gamma), x-\bar{x}\rangle+ \frac{1}{2\gamma}\|G^f_g(\bar{x};\gamma)\|^2.
\end{equation}
\end{lemma}


\begin{lemma}\label{growth}
Let $f(x)=g(Ex)$ with $g$ satisfying the properties \eqref{Lip} and \eqref{SC}, and denote $\hat{L}=L\|EE^T\|$. Then, we have
$$f(y)\geq f(x)+\langle \nabla f(x), y-x\rangle +\frac{\mu}{2}\|Ey-Ex\|^2, \quad\forall x, y\in\RR^n,$$
and
$$f(y)\leq f(x)+\langle \nabla f(x), y-x\rangle +\frac{\hat{L}}{2}\|y-x\|^2, \quad\forall x, y\in\RR^n.$$
\end{lemma}
\begin{proof}
 By applying integration, we derive $\forall x, y\in\RR^n$ that
\begin{subequations}
\begin{align}
&f(y)-f(x)-\langle \nabla f(x), y-x\rangle\\
=& \int_0^1 \langle \nabla f(x+\tau(y-x))-\nabla f(x), y-x\rangle d\tau\\
=&\int_0^1  \langle \nabla g(Ex+\tau\cdot E(y-x))-\nabla g(Ex), Ey-Ex\rangle d\tau.
\end{align}
\end{subequations}
Thus, by the strong convexity we get
$$f(y)-f(x)-\langle \nabla f(x), y-x\rangle \geq  \int_0^1  \mu \tau \|Ey-Ex\|^2 d\tau=\frac{\mu}{2}\|Ey-Ex\|^2,$$
and by the Lipschitz property and the Cauchy-Schwartz inequality we have
$$f(y)-f(x)-\langle \nabla f(x), y-x\rangle \leq  \int_0^1  L \tau \|Ey-Ex\|^2 d\tau\leq\frac{L\|EE^T\|}{2}\|y-x\|^2,$$
both of which complete the proof.
 \end{proof}

To deal with the gradient mapping of $f(x)=g(Ex)$, we need the following result which is motivated by Theorem 2.2.7 in \cite{nesterov2004introductory}.
 \begin{lemma}\label{nest2}
Let $f(x)=g(Ex)$ with $g$ satisfying the properties \eqref{Lip} and \eqref{SC}, and let $\gamma\geq \hat{L}(=L\|EE^T\|)$ and $\bar{x}\in\RR^n$. Then, for any $x\in Q$ we have
\begin{equation}\label{keyineq}
f(x)\geq f(x_Q(\bar{x};\gamma))+ \langle G^f_Q(\bar{x};\gamma), x-\bar{x}\rangle+ \frac{1}{2\gamma}\|G^f_Q(\bar{x};\gamma)\|^2+\frac{\mu}{2}\|Ex-E\bar{x}\|^2.
\end{equation}
\end{lemma}
\begin{proof}
Denote $x_Q=x_Q(\bar{x};\gamma), G=G^f_Q(\bar{x};\gamma)$ and let
$$\phi(x)=f(\bar{x})+ \langle \nabla f(\bar{x}), x-\bar{x}\rangle +\frac{\gamma}{2}\|x-\bar{x}\|^2.$$
Then, $\nabla\phi(x)=\nabla f(\bar{x})+\gamma (x-\bar{x})$ and for any $x\in Q$ we have
$$\langle \nabla f(\bar{x})-G, x-x_Q\rangle =\langle \nabla \phi(x_Q), x-x_Q\rangle \geq 0.$$
With this inequality and by Lemma \ref{growth}, we derive that
\begin{subequations}
\begin{align}
f(x)- \frac{\mu}{2}\|Ex-E\bar{x}\|^2&\geq f(\bar{x})+\langle \nabla f(\bar{x}), x-\bar{x}\rangle\\
&=f(\bar{x})+ \langle \nabla f(\bar{x}), x_Q-\bar{x}\rangle +\langle \nabla f(\bar{x}), x-x_Q\rangle\\
&\geq f(\bar{x})+ \langle \nabla f(\bar{x}), x_Q-\bar{x}\rangle +\langle G, x-x_Q\rangle\\
&=\phi(x_Q)-\frac{\gamma}{2}\|x_Q-\bar{x}\|^2+\langle G, x-x_Q\rangle\\
&=\phi(x_Q)+\frac{1}{2\gamma}\|G\|^2+\langle G, x-\bar{x}\rangle\\
&\geq f(x_Q)+\frac{1}{2\gamma}\|G\|^2+\langle G, x-\bar{x}\rangle,
\end{align}
\end{subequations}
where the last inequality follows from $f(x)\leq \phi(x)$ since $\gamma\geq \hat{L}$. Hence, the desired result holds.
\end{proof}

\begin{lemma}[Hoffman's error bound, \cite{Hoffman1952approximate,robinson1973bounds,Mangasarian1987Lip}]\label{hoffman}
Let $E$ be an $m\times n$ matrix and $A$ be a $k\times n$ matrix, and let $b$ be a vector in $\RR^k$. Then, there exists a scalar $\theta>0$ depending on $E$ and $A$ only such that, for any $y$ satisfying $Ax\leq b$ and any $t^*\in \RR^m$ such that the linear system $Ex=t^*, Ax\leq b$ is consistent, there is a point $\bar{y}\in \{u: Eu=t^*, Au\leq b\}$ satisfying $\theta\|y-\bar{y}\|^2\leq \|Ey-t^*\|^2$.
\end{lemma}

\begin{lemma}\label{opt}
Let $\gamma>0$ and denote $\mathcal{X}=\arg\min_{x\in Q} f(x)$. Then $\bar{x}\in Q$  is optimal for \eqref{cp} if and only if $G^f_Q(\bar{x};\gamma)=0$
\end{lemma}
The proof is identical to the proof of  Lemma A.6 in \cite{wang2013iteration}. Hence we omit the arguments.


\end{document}